\documentclass{amsart}
\usepackage{graphicx, amsmath, amssymb, amsthm}
\usepackage{mathabx}
\usepackage{tikz-cd}
\usepackage{hyperref}
\usepackage{todonotes}
\usepackage[all]{xy}

\newcommand{\groupsupport}[2]{The {#1} author was supported by {#2}.}

\newcommand{\NSFFour}{NSF Grant DMS--1811189}

\newcommand{\MAHAddress}{University of California Los Angeles, Los Angeles, CA 90095}
\newcommand{\MAHemail}{\tt{mikehill@math.ucla.edu}}

\newcommand{\R}{{\mathbb R}}
\newcommand{\F}{{\mathbb F}}

\newcommand{\MU}{MU}
\newcommand{\MUR}{\MU_{\R}}
\newcommand{\MUG}{\MU^{((G))}}

\newcommand{\MGL}{MGL}

\newcommand{\BPGL}{BPGL}

\newcommand{\smashover}[1]{\underset{#1}{\wedge}}
\newcommand{\Boxover}[1]{\underset{#1}{\Box}}

\DeclareMathOperator{\Ext}{Ext}

\newcommand{\m}[1]{{\protect\underline{#1}}}

\newcommand{\mF}{\m{\F}}

\mathchardef\mhyphen=45

\numberwithin{equation}{section}

\newtheorem{theorem}{Theorem}[section]
\newtheorem{lemma}[theorem]{Lemma}
\newtheorem{corollary}[theorem]{Corollary}

\newtheorem{proposition}[theorem]{Proposition}

\newtheorem*{theorem*}{Theorem}
\newtheorem*{proposition*}{Proposition}

\theoremstyle{remark}
\newtheorem{remark}[theorem]{Remark}

\newtheorem{notation}[theorem]{Notation}

\theoremstyle{definition}

\newtheorem{definition}[theorem]{Definition}

\newcommand{\ABNSF}{NSF Grant DMS--1906227}

\newcommand{\xk}[1]{\bar{v}(#1)}

\newtheorem*{corollary*}{Corollary}

\title{Transchromatic extensions in motivic bordism}

\author[AB]{Agn\`es Beaudry}
\thanks{\groupsupport{first}{\ABNSF}}
\address{University of Colorado Boulder,
Boulder, CO, 80309}
\email{agnes.beaudry@colorado.edu}

\author[MAH]{Michael A.~Hill}
\thanks{\groupsupport{second}{\NSFFour}}
\address{\MAHAddress}
\email{\MAHemail}

\author[XDS]{XiaoLin Danny Shi}
\address{University of Chicago, 
Chicago, IL 60637}
\email{dannyshixl@gmail.com}

\author[MZ]{Mingcong Zeng}
\address{Mathematical Institute, Utrecht University, Utrecht, 3584 CD, the Netherlands}
\email{m.zeng@uu.nl}

\begin{document}
\begin{abstract}
We show a number of Toda brackets in the homotopy of the motivic bordism spectrum \(MGL\) and of the Real bordism spectrum \(MU_{\mathbb R}\). These brackets are ``red-shifting'' in the sense that while the terms in the bracket will be of some chromatic height \(n\), the bracket itself will be of chromatic height \((n+1)\). Using these, we deduce a family of exotic multiplications in the \(\pi_{(\ast,\ast)}MGL\)-module structure of the motivic Morava \(K\)-theories, including non-trivial multiplications by \(2\). These in turn imply the analogous family of exotic multiplications in the \(\pi_{\star}MU_\mathbb R\)-module structure on the Real Morava \(K\)-theories.
\end{abstract}
\maketitle

\section{Introduction}
Complex bordism has played a fundamental role in stable homotopy since the 1960s. Work of Quillen connected complex bordism to formal groups, and this gives rise to the chromatic approach to stable homotopy theory. Building on Atiyah's Real \(K\)-theory \cite{AtiyahKR}, which can be viewed as Galois descent in families, Fujii and Landweber defined Real bordism \cite{Fujii, Landweber}. This theory plays an analogous role in \(C_{2}\)-equivariant homotopy theory that ordinary complex bordism does classically, and a detailed exploration of it was carried out by Hu--Kriz \cite{HuKriz}. 

The Real bordism spectrum has proven central to understanding classical chromatic phenomena. By the Goerss--Hopkins--Miller theorem, the Lubin--Tate spectra \(E_{n}\) are acted upon by the Morava stabilizer group, and in particular, they can be viewed as genuine \(G\)-equivariant spectra for any finite subgroup \(G\). Work of Hahn and the third author proved that at the prime \(2\), there is a Real orientation of all of the Lubin--Tate spectra, and for a finite subgroup \(G\) of the Morava stabilizer group that contains \(C_{2}\), this extends to a \(G\)-equivariant map
\[
\MUG=N_{C_{2}}^{G}\MUR\to E_{n},
\]
where \(\MUG\) is the norm of \(\MUR\) \cite{HahnShi}, introduced by Hill--Hopkins--Ravenel in the solution to the Kervaire invariant one problem \cite{HHR}. This has turned questions about computations with the Lubin--Tate theories into questions about computations with the norms of \(\MUR\) and its quotients (see, for example \cite{BHSZOne}).

In motivic homotopy over \(\mathbb R\), there is a beautifully parallel story. The role of complex bordism is played by the spectrum \(\MGL\). Just as in classical and equivariant homotopy theory, \(\MGL\) is a fundamental object of study in motivic homotopy, providing not only a chromatic filtration but also a way to understand the motivic homotopy sheaves of the sphere spectrum via Voevodsky's slice filtration.

Voevodsky's slice filtration is an analogue of the Postnikov tower, where instead of killing all maps of spheres of a particular degree, we kill off all maps out of sufficiently many smash powers of \(\mathbb P^{1}\). Applied to the algebraic \(K\)-theory spectrum \(KGL\) (an \(\MGL\)-module spectrum), this yields the motivic cohomology to algebraic \(K\)-theory spectral sequence considered by Friedlander--Suslin, Voevodsky, and others (See \cite{FriedSus, VoeMotSS}). Hopkins and Morel generalized this, showing that the slices of \(\MGL\) are suspensions of the spectrum representing motivic cohomology, and Hoyois provided a careful treatment and generalization of this result \cite{HopHoyMorel}. Work of Levine further connected the slice filtration to \(\MGL\), showing that the slice filtration for the sphere can be built out of the slice filtrations of the Adams--Novikov resolution based on \(\MGL\) \cite{LevineANSS}, and for the latter, the Hopkins--Hoyois--Morel result describes all of the starting pieces.

The spectrum \(\MGL\) is a commutative ring spectrum, so in addition to a commutative multiplication on the bigraded homotopy groups, we have higher operations and products like Toda brackets. In this paper, we produce a family of Toda brackets in the homotopy groups of \(\MGL\) which describe unexpected trans-chromatic phenomena. Recall that there is a canonical map 
\[
\pi_{2\ast}\MU\to\pi_{(2\ast,\ast)}\MGL
\]
classifying the canonical group law for the motivic orientation (see \cite{BorghesiMoravaK} or \cite{MotivicLandweber}), and hence we have associated to the chromatic classes \(v_{n}\in\pi_{2^{n+1}-2}\MU\) the motivic classes \(\bar{v}_{n}\in \pi_{(2^{n+1}-2, 2^n-1)}MGL\). Recall also that we have a canonical element \(\rho\in\pi_{(-1,-1)}S^{0}\) corresponding to the unit \(-1\in \R^{\times}\), by Morel's computation of the motivic zero stem \cite{MorelZero}. Finally, let \(I\) be the kernel of the map
\[
\pi_{\ast,\ast} \MGL\to \pi_{\ast,\ast} H\F_2
.\]

\begin{theorem*}
For all \(n\geq 0\) and \(k\geq 0\), in the motivic homotopy of \(\MGL\) over \(\R\), we have 
\[
\rho^{2^{n+1}+k}\bar{v}_{n+1}=\rho^{k}\langle\bar{v}_{n},\rho^{2^{n+1}-1},\bar{v}_{n}\rangle \mod I^2
.\]
\end{theorem*}

If we also consider the shifts of \(\bar{v}_n\) to other weights, using some of the other motivic lifts  \(\bar{v}_n(b)\) of the chromatic classes (the definitions of which we recall below), then we have even longer transchromatic connections.

\begin{theorem*}
For all \(n,j\geq 0\) and \(k\geq 0\), in the motivic homotopy of \(\MGL\) over \(\R\), 
\[
\rho^{2^{n+j+2}-2^{n+1}+k}\bar{v}_{n+j+1}= \rho^k\langle \bar{v}_n(2^j-1),\rho^{2^{n+1}-1},\bar{v}_n\rangle \mod I^2
.\]
\end{theorem*}

These transchromatic shifts imply a surprising number of hidden extensions in very naturally occurring quotients like the motivic Morava \(K\)-theories, the definitions of which we recall below in Section~\ref{sec:KGL}. Put in a pithy way, killing \(\bar{v}_{n}\) without also killing \(\bar{v}_{n+1}\) does not fully kill \(\bar{v}_{n}\):
\begin{theorem*}
For all \(n\geq 1\), for all \(0\leq k\leq n\) and for all \(b\geq 0\), in the homotopy of  motivic Morava \(K\)-theory  \(K_{GL}(n)\), there are nontrivial multiplications by \(\bar{v}_{k}(b)\).
\end{theorem*} 
Note that even though the brackets had ambiguity caused by higher Adams filtration, these statements do not.

There is a natural functor from motivic spectra over \(\R\) to \(C_{2}\)-equivariant spectra, extending the functor ``take complex points of a variety defined over \(\R\)'', and this takes \(\MGL\) to the Real bordism spectrum \(\MUR\) \cite{HuKrizMot}. Moreover, as studied by Hu--Kriz, Hill, and Heard, this connects the motivic slice filtration to Dugger's \(C_{2}\)-equivariant slice filtration \cite{HuKrizMot}, \cite{HillRhoBockstein}, \cite{HeardSlice}, \cite{DuggerKR}. 

This functor takes \(\rho\) to the Euler class  \(a_{\sigma}\in \pi_{-\sigma}^{C_2}\MU_\mathbb{R}\) and the copy of the Lazard ring in the homotopy of \(\MGL\) to the copy of the Lazard ring described by Araki in the homotopy of \(\MUR\) \cite{Araki}. In particular, the motivic classes \(\bar{v}_n\) are sent to the equivariant classes \(\bar{v}_n \in \pi_{(2^n-1)\rho_2}^{C_2}\MUR\), where \(\rho_2 = 1+\sigma\) is the regular representation of \(C_2\). The classes \(\bar{v}_n(b)\) are sent to the classes \(\bar{v}_nu_{2\sigma}^{2^nb}\), using the notation of \cite{HHR}. We therefore deduce the equivariant versions of these transchromatic phenomena. We spell these out for those more familiar with the equivariant literature.

\begin{corollary*}
For all \(n\geq 0\), in the \(RO(C_{2})\)-graded homotopy of \(\MUR\), we have an inclusions modulo \(I^2\), for \(I\) the kernel of \(\pi_{\star}MU \to \pi_{\star}H\mF_2\):
\[
a_\sigma^{2^{n+j+2}-2^{n+1}}\bar{v}_{n+j+1}\in \langle \bar{v}_n(2^j-1),a_\sigma^{2^{n+1}-1},\bar{v}_n\rangle,
a_{\sigma}^{2^{n+1}}\bar{v}_{n+1}\in\langle\bar{v}_{n},a_{\sigma}^{2^{n+1}-1},\bar{v}_{n}\rangle.
\]
\end{corollary*}

In general, the indeterminacy of these brackets may be larger in the \(C_2\)-equivariant context. We do not address this here. 
Using the Mackey functor structure on homotopy, we can actually identify the interesting element in the bracket as the unique non-zero element in the kernel of the restriction map.

There are also similar extensions in Hu--Kriz's Real Morava \(K\)-theories  \cite{HuKriz}.
\begin{theorem*}
For all \(n\geq 1\) and for all \(0\leq k\leq n\) and \(b\geq 0\), in the \(RO(C_{2})\)-graded homotopy of the Real Morava \(K\)-theory spectrum \(K_{\R}(n)\), there are nontrivial multiplications by \(\bar{v}_{k}(b)\).
\end{theorem*} 
For \(n=1\), this recovers the classical observation that multiplication by \(2\) is not identically zero in \(KO/2\). For larger \(n\) and \(k\), these seem unknown.

\subsection*{Acknowledgements}
We thank Dan Isaksen for some very helpful conversations related to Massey products and to steering us towards a much more direct proof of our results. We thank Nitu Kitchloo and Steve Wilson for interesting discussions of Real \(K\)-theory and Paul Arne {\O}stv{\ae}r for his help with the motivic slice spectral sequence for \(\MGL\) and \(k_{GL}(n)\).  We thank Irina Bobkova, Hans-Werner Henn and Viet Cuong Pham who made an observation which inspired our study of these brackets. We thank Mike Hopkins for several helpful conversations. We also thank the anonymous referee for their comments and suggestions.

\section{Non-trivial brackets in \texorpdfstring{\(\pi_{\ast,\ast}\MGL\)}{MGL}}

\subsection{The homotopy of \texorpdfstring{\(\MGL\)}{MGL}}
Throughout, we work with motivic spectra over \(\mathbb{R}\). We implicitly complete all spectra \(2\). Just as classically, we have a splitting of \(\MGL\) into various suspensions of a motivic ring spectrum \(\BPGL\), and since this is smaller, we will mainly work with it.  Then bigraded homotopy ring of \(\BPGL\) has been completely determined \cite{HillRhoBockstein}, \cite{Kylling}. We very quickly recall the answer here, using the description from \cite{HillRhoBockstein}.  

As an algebra, it is generated by the classes
\[
\bar{v}_n(b)=\Big[\bar{v}_n \tau^{2^{n+1}b}\Big]\in\pi_{(2^{n+1}-2, 2^n-1-2^{n+1}b)}\BPGL,
\]
for all \(n\geq 0\) and \(b\geq 0\), together with the class 
\[
\rho\in\pi_{(-1,-1)}\BPGL.
\]
Here, the brackets indicate how the classes arise in the \(\rho\)-Bockstein spectral sequence. We will also normally use \(\bar{v}_n\) for \(\bar{v}_n(0)\).

There are relations which reflect the underlying products with \(\tau\): if \(n\geq m\), then 
\[
\bar{v}_m(b)\cdot \bar{v}_n(c)=\bar{v}_m(b+2^{n-m}c)\cdot \bar{v}_n.
\]
We also have relations involving \(\rho\): for all \(n\) and \(b\),
\[
\rho^{2^{n+1}-1}\bar{v}_n(b)=0.
\]

The homotopy is actually largely concentrated in a particular bidegree sector. With the exception of the subalgebra generated by \(\rho\), the first (i.e. ``topological'') dimension is positive. 

\begin{proposition}\label{prop:NonNegFirstDegree}
Outside of the subalgebra spanned by \(\rho\), all elements have non-negative first coordinate. Moreover, the only generators with first coordinate zero are
\[
\rho^{2^{n+1}-2}\bar{v}_n(b),
\]
with \(b\) arbitrary. The products of any of these is zero.
\end{proposition}
\begin{proof}
Of the listed algebra generators, only \(\rho\) has a negative first coordinate, and the bidegree of \(\rho^j \bar{v}_n(b)\) is
\[
|\rho^j \bar{v}_n(b)|=\big((2^{n+1}-2)-j,(2^n-1)-j-2^{n+1}b\big).
\]
Since the only non-zero values correspond to \(0\leq j\leq 2^{n+1}-2\), we see that the first coordinate is always non-negative. This gives the first part.

For the second part, note that it is the last \(\rho\) power that gives a zero first coordinate. Since \(\rho\) times this is zero, we deduce that the product of any of these elements with first coordinate zero is zero. \end{proof}

Just as classically, the classes \(\bar{v}_n(b)\) depended on a choice of coordinate for the underlying formal group. However, we have a surprising invariance after multiplication by certain powers of \(\rho\).

\begin{lemma}\label{lem:RhoUniqueness}
	For any \(n\) and \(b\), if \(\bar{p}\) is an element in \(\pi_{\ast,\ast}\BPGL\) such that
	\[
		\bar{p}\equiv \bar{v}_n(b)\mod I^2,
	\]
	then for all \(k\geq (2^{\lfloor\tfrac{n}{2}\rfloor+1}-1)\),
	\[
		\rho^k\bar{p}=\rho^k\bar{v}_n(b).
	\]
\end{lemma}
\begin{proof}
While there are many monomials in the \(\bar{v}_i(a)\) in topological degree \((2^{n+1}-2)\), 
the only monomials of degree greater than one which survive even a single \(\rho\)-multiplication are the monomials in the \(\bar{v}_i(a)\) with \(i<n\). The order of the \(\rho\)-torsion on a monomial is governed by the smallest subscript present, so if we multiply by some \(\rho\)-power larger than any smallest subscript in a monomial in this degree, we annihilate any element of \(I^2\) in this degree. Our choice of \(k\) accomplishes this.
\end{proof}

\subsection{Toda brackets in the homotopy of \texorpdfstring{\(\MGL\)}{MGL}}

The relations 
\[
\rho^{2^{n+1}-1}\bar{v}_n(k)=0
\]
show that we can form many Toda brackets in the homotopy of \(\MGL\) and \(\BPGL\).

\begin{proposition}\label{prop:SwappingOrder}
	Let \(n\), \(m\), \(b\), and \(c\) be non-negative and let \(r\), \(s\), and \(t\) be non-negative integers such that
	\[
		s+r\geq 2^{n+1}-1\text{ and } s+t\geq 2^{m+1}-1.
	\]
	Then we have
	\[
		\langle \rho^r\bar{v}_m(b),\rho^s,\rho^t\bar{v}_n(c)\rangle = \pm \langle \rho^t\bar{v}_n(c),\rho^s,\rho^r\bar{v}_m(b)\rangle.
	\]
\end{proposition}
\begin{proof}
Kraines showed that we have a kind of symmetry invariance for higher brackets \cite[Theorem 8]{KrainesHigher}, up to a sign, for Massey products. The proof goes through without change for Today brackets.
\end{proof}

\begin{remark}
		Whenever \(r+s, s+t\geq (2^{n+1}-1)\), we can also form the brackets 
	\(
		\langle \rho^r,\rho^s\bar{v}_n,\rho^t\rangle
	\).
 For degree reasons, these are zero with zero indeterminacy.
\end{remark}

One of the most useful parts of these brackets is that the indeterminacy is easily controlled. Because of Proposition~\ref{prop:SwappingOrder}, we may without loss of generality pick an ordering on the heights of chromatic classes.

\begin{theorem}\label{thm:Indeterminacy}
Fix non-negative numbers \(m\geq n\), \(b\), and \(c\).  If \(r\), \(s\), and \(t\) are non-negative integers such that 
\[
r+s\geq (2^{m+1}-1)\text{ and } s+t \geq (2^{n+1}-1),
\]
then the indeterminacy of 
\[
\langle \rho^r\bar{v}_m(b),\rho^{s},\rho^t\bar{v}_n(c)\rangle
\]
is nonzero in one case: when \(t=0\), \(r+s=2^{m+1}-1\), and \(m>0\). In this case, the indeterminacy is
\[
    \mathbb Z_{2}\cdot\bar{v}_0\big((1+2b)2^{m-1}\big)\bar{v}_n(c).
\] 
\end{theorem}
\begin{proof}
The indeterminacy of the bracket \(\langle\rho^r\bar{v}_m(b),\rho^s,\rho^t\bar{v}_n(c)\rangle\) is the subgroup
\[
\rho^r\bar{v}_m(b)\pi_{(x_t,y_t)}\BPGL+
\rho^t\bar{v}_n(c)\pi_{(x_r,y_r)}\BPGL,
\]
where 
\[
(x_t,y_t)=|\rho^{s+t}\bar{v}_n(c)|+(1,0)=\big(2^{n+1}-1-s-t,2^n-1-s-t-2^{n+1}c\big)
\]
is the degree of the choices of null-homotopy of \(\rho^{s+t}\bar{v}_n(c)\), and similarly for \((x_r,y_r)\). 

By our assumptions on \(r\), \(s\), and \(t\), the first coordinate is always non-positive. 
By Proposition~\ref{prop:NonNegFirstDegree} if 
\[
r+s>(2^{m+1}-1)\text{ and }s+t>(2^{n+1}-1),
\]
then we have 
\[
\pi_{(x_r,y_r)}\BPGL=\pi_{(x_t,y_t)}\BPGL=0,
\]
and hence, we have no indeterminacy. We need only consider the cases that \((r+s)=(2^{m+1}-1)\) or \((s+t)=(2^{n+1}-1)\).

There is an obvious symmetry here in \(m\) and \(n\), so it suffices to understand the case \(r+s=(2^{m+1}-1)\). In this case,
\[
(x_r,y_r)=\big(0,-2^m-2^{m+1}b\big).
\]
Proposition~\ref{prop:NonNegFirstDegree} shows that the only classes with zero first coordinate are the classes \(\rho^{2^{k+1}-2}\bar{v}_k(a)\), which is in bidegree
\[
|\rho^{2^{k+1}-2}\bar{v}_k(a)|=\big(0,1-2^k(1+2a)\big).
\]
We are therefore looking for all pairs \((k,a)\) such that 
\[
1-2^k(1+2a)=-2^m(1+2b).
\]

If \(m>0\), then we must have \(k=0\) and \(a=2^{m-1}(1+2b)\), and this corresponds to the class
\[
\bar{v}_0\big((1+2b)2^{m-1}\big).
\]
This generates a \(\mathbb Z_{2}\), and hence the contribution to the indeterminacy is the subgroup generated by  
\[
\bar{v}_0\big((1+2b)2^{m-1}\big)\cdot \rho^t\bar{v}_n(c).
\]
If \(t>0\), then this is automatically zero (since \(\rho\bar{v}_0(a)=0\)), so if \(r+s=(2^{m+1}-1)\) and \(t>0\), then we have no contribution to the indeterminacy. On the other hand, if \(t=0\), then we have a contribution:
\[
\mathbb Z_{2}\cdot \bar{v}_0\big((1+2b)2^{m-1}\big)\bar{v}_n(c)=
\mathbb Z_{2}\cdot \bar{v}_0\big(2^{m-1}+2^mb+2^{n}c\big)\bar{v}_n.
\]
Note that if \(t=0\), then the condition that \(s+t\geq (2^{n+1}-1)\) implies that \(s\geq (2^{n+1}-1)\). This with the condition that \(r+s=(2^{m+1}-1)\) in particular shows the condition \(m\geq n\). 

Now let \(m=0\). We find \(k=\big(\nu_2(1+b)+1\big)\), for \(\nu_2\) the 2-adic valuation, and 
\[
a=\frac{2^{1-k}(1+b)-1}{2}.
\]
Note also that \(k\geq 1\), so the contribution to the indeterminacy is 
\[
\mathbb Z/2 \cdot \rho^{2^{k+1}-2}\bar{v}_k(a)\rho^t\bar{v}_n(c).
\]
Again, if \(t>0\), then this is automatically zero, since 
\(
\rho^{2^{k+1}-1}\bar v_k(a)=0.
\) 
If \(t=0\), then the conditions \(r+s=2^{0+1}-1=1\) and \(s+t=(2^{n+1}-1)\) imply that in fact, 
\(n=0\) as well. Since \(k\geq 1\), \((2^{k+1}-2)\geq 1\), and hence
\[
\rho^{2^{k+1}-2}\bar{v}_k(a)\bar{v}_0(c)=0.
\]
Thus in all cases, the contribution to indeterminacy here is zero.
\end{proof}

This vanishing on indeterminacy actually lets us tightly connect the various brackets for a fixed \(m\), \(n\), \(b\), and \(c\) via juggling formulae for Toda brackets.

\begin{proposition}\label{prop:Juggling}
	Fix non-negative numbers \(m\geq n\), \(b\), and \(c\), and non-negative integers \(r\), \(s\), and \(t\) such that
	\[
	(r+s)\geq (2^{m+1}-1)\text{ and }(s+t)\geq (2^{n+1}-1).
	\]
	Then multiplication by \(\rho\) links the brackets:
	\[
		\rho\cdot\langle \rho^r\bar{v}_m(b),\rho^s,\rho^t\bar{v}_n(c)\rangle=
		\langle \rho^{r+1}\bar{v}_m(b),\rho^s,\rho^t\bar{v}_n(c)\rangle.
	\]
	We can also internally juggle:
	\[ 
		\langle \rho^{r+1}\bar{v}_m(b),\rho^s,\rho^t\bar{v}_n(c)\rangle = 
		\langle \rho^r\bar{v}_m(b),\rho^{s+1},\rho^t\bar{v}_n(c)\rangle = 
		\langle \rho^r\bar{v}_m(b),\rho^s,\rho^{t+1}\bar{v}_n(c)\rangle.
	\]
\end{proposition}
\begin{proof}
	Both of these are analogues of so-called ``juggling'' formulae for Massey products. These go through without change for Toda brackets. 
	For the first, Kraines showed that we have an inclusion
	\[
		a\langle b,c,d\rangle\subseteq (-1)^{q}\langle ab,c,d\rangle,
	\]
	where \(q\) is the degree of \(a\) \cite[Theorem 6]{KrainesHigher}. Theorem~\ref{thm:Indeterminacy} shows that the ``larger'' bracket never has any indeterminacy, giving the result, since \(\rho\) is simple \(2\)-torsion in \(\MGL\).
	
	We have a similar juggle for moving internal products through brackets, using the matric Massey juggling due to May \cite[Corollary 3.4]{MayMatric}:
	\[
		\langle ab,c,d\rangle \subseteq \pm \langle a,bc,d\rangle.
	\]
	By the first part of the proposition, the bracket 
	\[
		\langle \rho^{r+1}\bar{v}_m(b),\rho^s,\rho^t\bar{v}_n(c)\rangle
	\]
	is in the image of multiplication by \(\rho\) and hence the signs do not matter. The conditions on \(r\), \(s\), and \(t\) also guarantee that neither bracket in 
	\[
		\langle \rho^{r+1}\bar{v}_m(b),\rho^s,\rho^t\bar{v}_n(c)\rangle \subseteq
		\langle \rho^r\bar{v}_m(b),\rho^{s+1},\rho^t\bar{v}_n(c)\rangle
	\]
	has any indeterminacy. The last equality follows by symmetry.
\end{proof}

\subsection{Identifying brackets via the Adams spectral sequence}
We can describe what elements arise in these Toda brackets using the Adams spectral sequence. Just as classically, the homology of \(\BPGL\) is cotensored up along a quotient Hopf algebroid. Voevodsky computed the dual Steenrod algebra over \(\mathbb R\), showing that as a Hopf algebroid over 
\[
\mathbb M_2=\F_2[\rho,\tau],
\] 
the motivic homology of a point \cite{VoevodskyMod2}, we have
\[
\mathcal A_{\ast,\ast}=\mathbb M_2[\xi_1,\dots][\tau_0,\dots]/\big(\tau_i^2+\rho\tau_{i+1}+(\tau+\rho\tau_0)\xi_{i+1}\big).
\]
The element \(\rho\) is primitive, the left unit on \(\tau\) is the obvious inclusion, and the right unit is \(\tau+\rho\tau_0\). The coproducts on the \(\xi_i\) and the \(\tau_i\) are the classical ones \cite{VoevodskySteenrod}. 

Let 
\[
\mathcal E(\infty)=\mathbb M_2[\tau_0,\tau_1,\dots]/(\tau_i^2-\rho\tau_{i+1})
\]
be the quotient of the motivic dual Steenrod algebra by the ideal generated by the \(\xi_i\)s.
This is a Hopf algebroid under the motivic dual Steenrod algebra, and the generators \(\tau_i\) are now primitive. In particular, all of the interesting behavior in \(\Ext\) is determined by the left and right units on \(\tau\) (since \(\rho\), being in the Hurewicz image, is necessarily primitive). 

Just as classically, this Hopf algebroid is related to the homology of \(\BPGL\), via the cotensor product. Recall that if \(\Gamma\) is a Hopf algebroid, \(M\) is a right \(\Gamma\)-comodule with coproduct \(\psi_M\), and \(N\) a left \(\Gamma\)-comodule with coproduct \(\psi_N\), then the cotensor product of \(M\) and \(N\) is defined by an equalizer diagram
\[
\begin{tikzcd}
{M\Boxover{\Gamma} N}
	\ar[r]
	&
{M\otimes N}
	\ar[r, shift right=.5ex, "\psi_M\otimes 1"']
	\ar[r, shift left=.5ex, "1\otimes \psi_N"]
	&
{M\otimes\Gamma\otimes N}.
\end{tikzcd}
\]

\begin{theorem}[{\cite{BorghesiMoravaK}, \cite{OrmsbyExt}}]
We have an isomorphism of \(\mathcal A_\star\)-comodule algebras
\[
H_\ast(\BPGL;\F_2)\cong \mathcal A_\star\Boxover{\mathcal E(\infty)}\mathbb M_2.
\]
\end{theorem}

Note that there are no \(\rho\)-torsion elements in either \(\mathbb M_2\) or \(\mathcal E(\infty)\), so we can divide uniquely by \(\rho\) various \(\rho\)-divisible elements.

\begin{theorem}[{\cite[Cor 5.2 and Thm 5.3]{HillRhoBockstein}}]\label{thm:ChangeOfRings}
The \(E_2\)-term of the Adams spectral sequence computing the homotopy of \(\BPGL\) is 
\[
E_2^{s,\star}=\Ext_{(\mathbb M_2,\mathcal E(\infty))}^{s,\star}(\mathbb M_2,\mathbb M_2).
\]
The elements \(\bar v_n(b)\) are detected by  
\[
\frac{d(\tau^{(1+2b)2^n})}{\rho^{2^{n+1}-1}}=\frac{\tau^{(1+2b)2^n}+(\tau+\rho\tau_0)^{(1+2b)2^{n}}}{\rho^{2^{n+1}-1}}\in
\Ext^{1,(2^{n+1}-1,2^{n}-1-2^{n+1}b)},
\]
where \(d\) is the cobar differential, and the class \(\rho\) by itself in \(\Ext^{0,(-1,-1)}\). The spectral sequence collapses with no exotic extensions.
\end{theorem}

The Hopf algebroid \(\big(\mathbb M_2,\mathcal E(\infty)\big)\) is computationally very simple: we have a primitive polynomial generator \(\rho\) and a second, non-primitive element \(\tau\). The ring \(\mathcal E(\infty)\) is not polynomial on
\[
\tau_0=\frac{\eta_L(\tau)-\eta_R(\tau)}{\rho}.
\] 
Instead, we have a kind of ``\(\rho\)-divided power algebra''. The real power of Theorem~\ref{thm:ChangeOfRings} is that all of the generators of \(\Ext^1\) can be realized as \(\rho\)-fractional multiples of the usual cobar differential on powers of \(\tau\). That will allow us to easily compute Massey products in \(\Ext\).

It is helpful in what follows to blur the chromatic heights of the elements, focusing instead on the powers of \(\tau\) and their differentials. 

\begin{notation}
For each \(k\geq 1\), let \(m_k=2^{\nu_2(k)+1}-1\) for  \(\nu_2(k)\)  the \(2\)-adic valuation of \(k\) and let
\[
\xk{k}=\frac{\eta_L(\tau)^k-\eta_R(\tau)^k}{\rho^{m_k}}.
\]
Note that \(\bar v_n(b)\) is detected by the element \(\bar{v}\big((1+2b)2^n\big)\).
\end{notation}
It is not hard to show that \(\rho^{m_k}\) is the largest power of \(\rho\) which divides the cobar differential on \(\tau^k\). 

\begin{theorem}\label{thm:HopfAlgBracket}
Let \(k\) and \(\ell\) be non-negative, and let \(r\), \(s\), and \(t\) be natural numbers such that
\[
s+r\geq m_k\text{ and }s+t\geq m_{\ell}.
\]
Then in \(\Ext\), we have an inclusion
\[
\rho^{m_{k+\ell}-m_{k}-m_{\ell}+r+s+t}\xk{k+\ell}\in\big\langle\rho^r \xk{k},\rho^{s},\rho^t\xk{\ell}\big\rangle.
\]
\end{theorem}
\begin{proof}
There are preferred null-homotopies of \(\xk{k}\rho^{r+s}\) and \(\xk{\ell}\rho^{s+t}\):
\[
\rho^{r+s-m_k}\tau^k\mapsto \rho^{r+s}\xk{k}\text{ and }\rho^{s+t-m_{\ell}}\tau^{\ell}\mapsto\rho^{s+t}\xk{\ell}.
\]
The bracket in question then contains
\[
\rho^{s+t+r-m_\ell}\xk{k}\eta_R(\tau^{\ell})+\rho^{r+s+t-m_k}\tau^{k}\xk{\ell}.
\]
Unpacking the \(\rho\)-fractions giving \(\xk{k}\) and \(\xk{\ell}\) and recalling that the cobar differential is a bimodule derivation, we see that this particular element is 
\[
\rho^{m_{k+\ell}-m_k-m_\ell+r+s+t}\xk{k+\ell}.\qedhere
\]
\end{proof}

\begin{remark}
These same relations will hold for any similar Hopf algebroid with a primitive element playing the role of \(\rho\) and the differential on a class analogous to \(\tau\) involving only \(\rho\) divisibility.  For example, we see the same brackets involving the \(\alpha\)-family in the Miller--Ravenel--Wilson Chromatic Spectral Sequence approach to understanding the classical Adams--Novikov \(E_2\)-term \cite{MRW}.
\end{remark}

\begin{remark}
When \(n=0\), Theorem~\ref{thm:WeirdBracket} gives the formula
\[
\rho\eta\in\langle 2,\rho,2\rangle,
\]
since \(\rho\bar{v}_{1}\) is the ordinary, topological \(\eta\). Since \(2\rho=0\) in the homotopy of \(\MGL\) over \(\mathbb R\), this is recovering a universal classical formula in an \(A_\infty\)-ring spectrum that forming the balanced bracket with \(2\) gives \(\eta\) multiplication \cite{SagaveUniversal}. 
\end{remark}

\subsection{Convergence}

Moss's Convergence Theorem allows us to lift these Massey products in the Adams spectral sequence to Toda brackets in homotopy \cite{MossConvergence}. Normally in applying Moss's Convergence Theorem, we have to worry about so-called ``crossing differentials''. In our case, the Adams spectral sequence collapsed at \(E_2\) with no differentials. We deduce the following.

\begin{theorem}\label{thm:WeirdBracket}
Let \(b\) and \(c\) be nonegative numbers. Let \(r\), \(s\), and \(t\) be such that
\[
r+s,s+t\geq m_n=(2^{n+1}-1).
\]
Let \(j\) and \(d\) be such that we have
\[
1+b+c=(1+2d)2^j.
\]
Then, modulo elements of higher Adams filtration, we have an inclusion
\[
\rho^{(2^{n+j+2}-2^{n+2}+r+s+t+1)}\bar{v}_{n+j+1}(d)\in
\big\langle \rho^r\bar{v}_n(b),\rho^s,\rho^t\bar{v}_n(c)\big\rangle.
\]
\end{theorem}
As stated, the theorem has markedly limited appeal, since there are in general a great many classes in higher Adams filtration.
Using Proposition~\ref{prop:Juggling} and Lemma~\ref{lem:RhoUniqueness}, we can remove some ambiguity, since the Adams filtration here coincides with the monomial filtration in induced by \(I\). 

\begin{corollary}
Let \(b\) and \(c\) be nonnegative numbers, and let \(j\) and \(d\) be such that
\[
1+b+c=(1+2d)2^j.
\]
Then for any \(k\geq (2^{\lfloor\tfrac{n}{2}\rfloor+1}-1)\), we have
\[
\rho^{2^{n+k+2}-2^{n+2}+k+1}\bar{v}_{n+j+1}(d)=
\langle \bar{v}_n(b),\rho^{2^{n+1}-1+k},\bar{v}_n(c)\rangle.
\]
\end{corollary}

\section{Application to Morava \texorpdfstring{\(K\)}{K}-theories}\label{sec:KGL}
The transchromatic brackets give us some surprising consequences for the action of \(\pi_{(\ast,\ast)}\MGL\) on various quotients. Since \(\MGL\) is a commutative monoid, we have a good category of \(\MGL\)-modules. 
For a single element \(x_i\), we define the quotient by \(x_i\) via the cofiber sequence
\[
\Sigma^{|x_{i}|}\MGL\xrightarrow{x_{i}}\MGL\to\MGL/x_{i},
\]
and for a family of elements \(x_{1},x_{2}\dots\), we form
\[
\MGL/(x_{1},x_2,\dots)= \MGL/x_{1}\smashover{\MGL}  \MGL/x_{2}\smashover{\MGL} \dots.
\]

\begin{definition}
For each \(n\geq 0\), let 
\[
k_{GL}(n)=\BPGL/(\bar{v}_{0},\dots,\bar{v}_{n-1},\bar{v}_{n+1},\dots)
\]
be the \(n\)th connective motivic Morava \(K\)-theory. 

Let \(K_{GL}(n) = \bar v_n^{-1}k_{GL}(n)\) be motivic Morava \(K\)-theory. 
\end{definition}

Using the Hopkins--Hoyois--Morel determination of the slices of \(\MGL\) \cite{HopHoyMorel}, Levine--Tripathi determined the slices for \(k_{GL}(n)\) (and related quotients).

\begin{theorem}[{\cite[Cor. 4.6]{LevineTripathiQuotients}}]
The slice associated graded for \(k_{GL}(n)\) is
\[
Gr(k_{GL}(n))=\bigvee_{m\geq 0} \Sigma^{(2^{n+1}-2)m,(2^{n}-1)m}H\F_{2}.
\]
\end{theorem}
Moreover, the associated slice spectral sequence is a spectral sequence of modules over the slice spectral sequence for \(\BPGL\), which is spelled out in \cite{Kylling}. 
For degree reasons, the spectral sequence for \(k_{GL}(n)\) is even simpler; we have thrown away many of the classes which supported or could have supported differentials. 

\begin{theorem}[{\cite[Theorem 9.6]{Kylling}, \cite{YagitaAHSS}}]
The slice spectral sequence for \(k_{GL}(n)\) has \(E_{2}\)-term
\[
\mathbb F_{2}[\rho,\tau,\bar{v}_{n}]\{\iota\},
\]
where \(|\rho|=(-1,-1)\), where \(|\tau^{2}|=(0,-2)\), and where \(|\bar{v}_{n}|=(2^{n+1}-2,2^{n}-1)\).

As an \(\mathbb F_{2}[\rho,\bar{v}_{n},\tau^{2^{n+1}}]\)-module, the only non-trivial differentials are generated by 
\[
d_{2^{n+1}-1}(\tau^{a+2^{n}}\iota)=\rho^{2^{n+1}-1}\bar{v}_{n}\tau^{a}\iota,
\]
where \(0\leq a\leq 2^{n}-1\).
\end{theorem}

Away from the subalgebra \(\mathbb F_{2}[\rho]\), the localization map
\[
k_{GL}(n)\to K_{GL}(n)
\]
is injective. The latter is an \(\MGL[\bar{v}_{n}^{-1}]\)-module. In \(\MGL[\bar{v}_{n}^{-1}]\), the class \(\tau^{2^{n+1}}\)  is a permanent cycle and multiplication by this is well-defined in \(K_{GL}(n)\). By injectivity of the localization, we may therefore view everything as being a module over
\[
	\mathbb Z_{2}[\rho,\bar{v}_{n},\tau^{2^{n+1}}].
\]

\begin{corollary}\label{cor:ModuleGenerators}
As a module over \(\mathbb Z_{2}[\rho,\bar{v}_{n},\tau^{2^{n+1}}]\), the homotopy of \(k_{GL}(n)\) is generated by the classes \(\tau^{a}\iota\) for \(0\leq a\leq (2^{n}-1)\).
\end{corollary}

Although the spectral sequence collapses here, we have a surprising number of non-trivial multiplications by the \(\bar{v}_{k}\)-generators that we killed to form \(k_{GL}(n)\), including non-trivial multiplications by \(2\). These are all detected by our brackets, using the sparseness of the spectral sequence (and hence sparseness of the bigraded homotopy groups). 

We have a simple consequence of the module structure and the presence of \(\tau^{2^{n+1}}\): we need only check a small number of hidden extensions.

\begin{corollary}
We need only determine hidden extensions of the form \(\bar{v}_k(b)\tau^a\iota\) for \(0\leq b\leq (2^{n-k}-1)\).
\end{corollary}
\begin{proof}
Since there is a class \(\tau^{2^{n+1}}\), we have 
\(
\bar{v}_k(b+2^{n-k})=\bar{v}_k(b)\tau^{2^{n+1}}.\qedhere
\)
\end{proof}

\begin{remark}\label{rem:aboutfig}
While going through our analysis of extensions, the reader is encouraged to consult Figure~\ref{fig:kC23} which shows the slice associated graded for the bigraded homotopy groups of \(k_{GL}(3)\). The grading is the usual motivic bigrading. The circled classes are the generators as a module over \(\mathbb Z_{2}[\rho,\bar{v}_{n},\tau^{2^{n+1}}]\), and the notation is as follows:
\begin{itemize}
	\item The dotted lines indicate multiplication by \(\rho\). 
	\item The dashed blue lines indicate exotic multiplication by \(\bar{v}_{2}\), and the solid red lines indicate exotic multiplication by \(\bar{v}_{1}\). 
	\item The bullets \(\bullet\) and solid black squares \(\blacksquare\) both indicate copies of \(\mathbb F_{2}\); bullets complexify to zero while black squares complexify to the corresponding generator.
	\item The open circles \(\bigcirc\) are copies of \(\mathbb Z/4\) which complexify to \(\mathbb Z/2\). 
\end{itemize}
To avoid clutter, we do not draw the exotic multiplications by \(\bar{v}_k(b)\) for \(b>0\) and \(k<3\).
\end{remark}

\begin{figure}[ht]
\includegraphics[width=\textwidth]{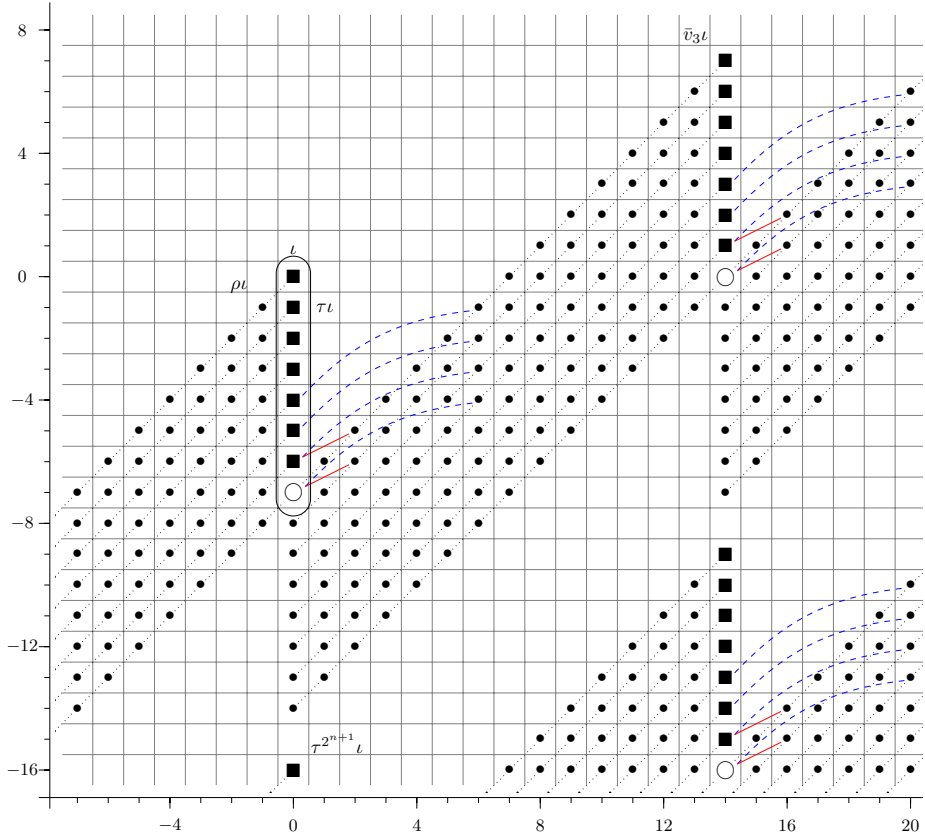}
\caption{The slice \(E_{\infty}\) term for \(k_{GL}(3)\), together with extensions. See Remark~\ref{rem:aboutfig} for the notation.} 
\label{fig:kC23}
\end{figure}

\begin{lemma}\label{lem:Sparseness}
Let \(a\) and \(b\) be non-negative integers and \(0\leq k<n\). Then there is at most one non-zero class that could be \(\bar{v}_k(b)\tau^a\iota\).

If
\[
a+2^k(1+2b)<2^n\text{ or }2^{n+1}-1<a+2^k(1+2b),
\]
then we have
\(
\bar{v}_k(b)\tau^a\iota=0.
\)
Moreover, the ideal \(I_n^2\) acts trivially.
\end{lemma}
\begin{proof}
The \(E_{\infty}\)-page of the slice spectral sequence looks like a quilt. Applying the degree sheer \((x',y')=(x,y-x)\), the homotopy groups are built out of rectangles of size \((2^{n+1}-1)\) by \((2^{n})\), with a single non-zero class in each degree. The generator of each rectangle is at the top right corner, given by \(\tau^{2^{n+1}m} \bar{v}_n^{\ell}\iota\) for \(m\geq 0\) and \(\ell\geq 0\). 

Since \(\bar{v}_{n}\) has sheered bidegree \(\big((2^{n+1}-2),-(2^{n}-1)\big)\), we see that the rectangle for \(\alpha\) overlaps with the one for \(\bar{v}_n\alpha\) in exactly one corner: the bottom right for \(\alpha\) and the top left for \(\bar{v}_{n}\alpha\). Multiplication by \(\tau^{2^{n+1}}\) adds in a rectangle with generator in degree 
\(
|\alpha|+(0,-2^{n+1}).
\) 
This does not intersect either the original rectangle or any of its \(\bar{v}_{n}\)-multiples. In particular, this proves the first claim, since almost all degrees have a single non-zero class in them, and the degrees with overlap have a single class in higher filtration, where any hidden extensions must land.

For the second and third claims, recall that the degree of \(\bar{v}_k(b)\tau^a\iota\) is
\[
|\bar{v}_k(b)\tau^a\iota|=\big(2^{k+1}-2,2^k-1-2^{k+1}b-a\big).
\]
Since the classes \(\tau^a\iota\) are at the rightmost edge of their rectangle, these extensions must show up in the rectangle from \(\bar{v}_n\iota\). The top edge of the rectangle for \(\bar{v}_n\iota\) is given by the classes \(\rho^m\bar{v}_n\iota\). The only class on that edge with first coordinate \((2^{k+1}-2)\) is \(
\rho^{2^{n+1}-2^{k+1}}\bar{v}_n\iota,
\)
which is in bidegree
\[
|\rho^{2^{n+1}-2^{k+1}}\bar{v}_n\iota|=\big(2^{k+1}-2,1-2^n+2^{k+1}-2\big).
\]
Similarly, the bottom edge of the rectangle on \(\bar v_n \iota\) is given by the classes \(\rho^m\bar{v}_n\tau^{2^n-1}\iota\).  Rearranging the second coordinates gives the desired vanishing region. 

Finally, note that \((2^{n+1}-2^{k+1})\geq (2^{j+1}-1)\) for any \(j\leq k<n\), so the target of \(\bar{v}_k(b)\)-multiplication annihilates \(\bar{v}_j(c)\). This is the third claim.
\end{proof}

\begin{theorem}\label{thm:kGLnDivisibility}
In the homotopy of \(k_{GL}(n)\), we have the following hidden \(\pi_{(\ast,\ast)}\BPGL\)-multiplications for \(0\leq k\leq (n-1)\):
\begin{itemize}
    \item for \(0\leq b\leq (2^{n-k-1}-1)\) and \(0\leq a\leq \big(2^{k}(1+2b)-1\big)\),
    \[
    \bar{v}_k(b) \tau^{2^{n}-2^{k}(1+2b)+a}\iota=\rho^{2^{n+1}-2^{k+1}}\bar{v}_n\tau^{a}\iota,
    \]
    \item and for \(2^{n-k-1}\leq b\leq (2^{n-k}-1)\), and \(0\leq a\leq \big(2^{n+1}-1-2^k(1+2b)\big)\),
    \[
    \bar{v}_k(b)\tau^a\iota=\rho^{2^{n+1}-2^{k+1}}\bar{v}_n\tau^{a+2^k(1+2b)-2^n}\iota.
    \]
\end{itemize}
\end{theorem}

\begin{proof}
These follow surprisingly quickly from Theorem~\ref{thm:WeirdBracket}. We note that since multiplication by \(\rho\) is injective on both the source and target of any extension involving \(\bar{v}_k(b)\) for \(k>0\), there is no ambiguity caused by the indeterminacy. There was no indeterminacy anyway for brackets involving \(\bar{v}_0(b)\). Since the ideal \(I_n^2\) acts trivially, we also have no ambiguity due to higher Adams filtration elements.

For \(0\leq b\leq (2^{n-k-1}-1)\), Theorem~\ref{thm:WeirdBracket} shows that we have a bracket
\[
\rho^{2^{n+1}-2^{k+1}}\bar{v}_n=\big\langle \bar{v}_k(b),\rho^{2^{k+1}-1},\bar{v}_k(2^{n-k-1}-1-b)\big\rangle,
\]
(where again, we ignore indeterminacy since it does not contribute). This gives us
\[
\rho^{2^{n+1}-2^{k+1}}\bar{v}_n\tau^a\iota=\big\langle \bar{v}_k(b),\rho^{2^{k+1}-1},\bar{v}_k(2^{n-k-1}-1-b)\big\rangle\tau^a\iota.
\]
Lemma~\ref{lem:Sparseness} shows that for \(a<2^k(1+2b)\), we have
\[
\bar{v}_k(2^{n-k-1}-1-b)\tau^a\iota=0,
\]
so we can shuffle the bracket, giving:
\[
\rho^{2^{n+1}-2^{k+1}}\bar{v}_n\tau^a\iota=\bar{v}_n(b)\big\langle\rho^{2^{k+1}-1},\bar{v}_k(2^{n-k-1}-1-b),\tau^a\iota\big\rangle.
\]
A degree check shows that the only possible value for the bracket is
\[
\big\langle\rho^{2^{k+1}-1},\bar{v}_k(2^{n-k-1}-1-b),\tau^a\iota\big\rangle=\tau^{2^n-2^{k}(1+2b)+a}\iota.
\]

For \(2^{n-k-1}\leq b\leq 2^{n-k}-1\), we instead will use brackets with \(\bar{v}_n(1)=\bar{v}_n\tau^{2^{n+1}}\), arguing somewhat indirectly. If we let 
\[
c=2^{n-k}+2^{n-k-1}-b-1,
\]
then we have
\[
\rho^{2^{n+1}-2^{k+1}}\bar{v}_n(1)=\big\langle\bar{v}_k(c),\rho^{2^{k+1}-1},\bar{v}_k(b)\big\rangle.
\]
If \(\bar{v}_k(b)\tau^a\iota=0\), then we can multiply \(\tau^a\iota\) by both sides and shuffle the bracket, getting
\[
\rho^{2^{n+1}-2^{k+1}}\bar{v}_n(1)\tau^a\iota=\bar{v}_k(c)\big\langle \rho^{2^{k+1}-1},\bar{v}_k(b),\tau^a\iota\big\rangle.
\]
The degree of the bracket is \(\big(0,-a-2^k(1+2b)\big)\). The bounds 
\[
0\leq a\leq 2^{n+1}-1-2^k(1+2b)
\]
and \(b\geq 2^{n-k-1}\) show that we have bounds
\[
-2^k-2^n\geq -a-2^k(1+2b)\geq 1-2^{n+1}.
\]
We are therefore at the leftmost edge of the ``quilt rectangle'' generated by \(\bar{v}_n\iota\). All of the elements with first coordinate zero here are annihilated by \(\rho\), so we reach a contradiction:
\[
0\neq\rho^{2^{n+1}-2^{k+1}+1}\bar{v}_n(1)\tau^a\iota=\bar{v}_k(c)\rho\big\langle \rho^{2^{k+1}-1},\bar{v}_k(b),\tau^a\iota\big\rangle=0.
\]
As always, there is a unique possibility for the value of \(\bar{v}_k(b)\tau^a\iota\), and checking degrees, we see it must be the listed one. 
\end{proof}

Specializing to the chromatic classes \(\bar{v}_k\), we have a series of non-trivial extensions.

\begin{corollary}
In the homotopy of \(k_{GL}(n)\), we have the following hidden \(\bar{v}_k\)-multiplications for \(0\leq k\leq (n-1)\) and \(0\leq a\leq (2^k-1)\):
\[
\bar{v}_k \tau^{2^{n}-2^{k}+a}\iota=\rho^{2^{n+1}-2^{k+1}}\bar{v}_n\tau^{a}\iota.
\]
\end{corollary}

Since \(\bar{v}_0\) detects multiplication by \(2\) in \(\MGL\)-modules, we have non-trivial additive extensions.  This resolves the question described in \cite[Remark 9.8]{Kylling}.

\begin{remark}
These extensions show just how far the quotient is from being a ring. One way to parse Theorem~\ref{thm:kGLnDivisibility} is that although we cone-off \(\bar{v}_k\) for \(0\leq k\leq n-1\), and hence we seem to kill all of the classes \(\bar{v}_k(b)\), we actually see non-trivial multiplications by all of the generators of \(\pi_{\ast,\ast}\BPGL/(\bar{v}_{n+1},\dots)\).
\end{remark}

Even though the indeterminacy for brackets may grow in the passage from motivic homotopy over \(\R\) to \(C_2\)-equivariant homotopy, the extensions we found are visible just in the homotopy. They will in particular go through without change. We indicate the result using the usual equivariant names.

\begin{corollary}
In the homotopy groups of \(K_{\R}(n)\), we have exotic \(\pi_\star\MUR\) multiplications for \(0\leq k\leq (n-1)\):
\begin{itemize}
    \item for \(0\leq b\leq (2^{n-k-1}-1)\), and \(0\leq a\leq \big(2^{k}(1+2b)-1\big)\),
    \[
    \bar{v}_k(b) u_{\sigma}^{2^{n}-2^{k}(1+2b)+a}\iota=a_\sigma^{2^{n+1}-2^{k+1}}\bar{v}_n u_\sigma^{a}\iota,
    \]
    \item and for \(2^{n-k-1}\leq b\leq (2^{n-k}-1)\), and \(0\leq a\leq \big(2^{n+1}-1-2^k(1+2b)\big)\),
    \[
    \bar{v}_k(b)u_\sigma^a\iota=a_\sigma^{2^{n+1}-2^{k+1}}\bar{v}_nu_\sigma^{a+2^k(1+2b)-2^n}\iota.
    \]
\end{itemize}
\end{corollary}

\bibliographystyle{plain}


\end{document}